\documentclass[11pt,oneside, letter]{article}

\usepackage{amsthm, amsmath, amssymb, amsfonts, graphicx, epsfig}
\usepackage{accents}

\usepackage{bbm}
\usepackage{mathrsfs}

\usepackage{epstopdf}

\usepackage{graphicx}
\usepackage[font=sl,labelfont=bf]{caption}
\usepackage{subcaption}

\usepackage{float}
\restylefloat{table}
\usepackage{slashbox}

\usepackage{enumerate}

\usepackage[colorlinks=true, pdfstartview=FitV, linkcolor=blue,
            citecolor=blue, urlcolor=blue]{hyperref}
\usepackage[usenames]{color}
\usepackage{cancel}

\usepackage{url}
\makeatletter\def\url@leostyle{%
 \@ifundefined{selectfont}{\def\UrlFont{\sf}}{\def\UrlFont{\scriptsize\ttfamily}}} \makeatother

\usepackage[margin=1.3in, letterpaper]{geometry}

\newtheorem{theorem}{Theorem}
\newtheorem{proposition}[theorem]{Proposition}

\newtheorem{corollary}[theorem]{Corollary}

\theoremstyle{definition}
\newtheorem{definition}[theorem]{Definition}
\newtheorem{example}[theorem]{Example}
\theoremstyle{remark}
\newtheorem{remark}[theorem]{Remark}

\numberwithin{equation}{section}
\numberwithin{theorem}{section}

\def\cB{\mathcal{B}}

\def\cF{\mathcal{F}}

\def\cN{\mathcal{N}}

\def\cT{\mathcal{T}}

\def\cX{\mathcal{X}}

\def\bE{\mathbb{E}}

\def\bP{\mathbb{P}}

\def\bR{\mathbb{R}}

\def\sF{\mathscr{F}}

\def\sH{\mathscr{H}}

\def\sU{\mathscr{U}}

\def\mH{\mathsf{H}}

\newcommand{\1}{\mathbbm{1}}            
\newcommand{\set}[1]{\{#1\}}            
\renewcommand{\mid}{\;|\;}              

\title{ \vspace{-1.5em} 
       Recursive Construction of Confidence Regions
}

\makeatletter
\def\and{%
  \end{tabular}%
  \begin{tabular}[t]{c}}%
\def\@fnsymbol#1{\ensuremath{\ifcase#1\or a\or b\or c\or
   d\or e\or f\or g\or h\or i\else\@ctrerr\fi}}
\makeatother

\author{
        Tomasz R. Bielecki\,\thanks{Department of Applied Mathematics, Illinois Institute of Technology
       \newline \hspace*{1.45em}  10 W 32nd Str, Building RE, Room 208, Chicago, IL 60616, USA
       \newline \hspace*{1.45em}  Emails: \url{tbielecki@iit.edu} (Bielecki), \url{tchen29@iit.edu} (Chen) and \url{cialenco@iit.edu} (Cialenco)
       \newline \hspace*{1.45em}  URLs: \url{http://math.iit.edu/\~bielecki} (Bielecki) and \url{http://math.iit.edu/\~igor} (Cialenco)
        \vspace{0.5em}} ,
\and
        Tao Chen\,\footnotemark[1] ,
\and
        Igor Cialenco\,\footnotemark[1]
        }

\date{ {\small  
First Circulated: May 18, 2016\\
This Version: May 15, 2017
}} %

\begin{document}

\maketitle


{\footnotesize
\begin{tabular}{l@{} p{350pt}}
  \hline \\[-.2em]
  \textsc{Abstract}: \ &
Assuming that one-step transition kernel of a discrete time, time-homogenous Markov chain model is parameterized by a  parameter $\theta\in \boldsymbol \Theta$, we derive a recursive (in time)  construction of confidence regions for the unknown parameter of interest, say $\theta^*\in \boldsymbol \Theta$. It is supposed that the observed data used in construction of the confidence regions is generated by a Markov chain whose transition kernel corresponds to $\theta^*$ .  The key step in our construction is derivation of a recursive scheme for an appropriate point estimator of $\theta^*$. To achieve this, we start by what we call the base recursive point estimator, using which we design a quasi-asymptotically linear recursive point estimator (a concept introduced in this paper). For the latter estimator we prove its weak consistency and asymptotic normality. The recursive construction of confidence regions is needed not only for the purpose of speeding up the computation of the successive confidence regions, but, primarily, for the ability to apply the dynamic programming principle in the context of robust adaptive stochastic control methodology.
\\[0.5em]
\textsc{Keywords:} \ &  recursive confidence regions; stochastic approximation; recursive point estimators; statistical inference for Markov chains; ergodic processes; quasi-asymptotically linear estimator. \\
\textsc{MSC2010:} \ & 62M05, 62F10, 62F12, 62F25, 60J05, 60J20. \\[1em]
  \hline
\end{tabular}
}

\section{Introduction}
Suppose that a set of dynamic probabilistic models is selected and that it is parameterized in terms of a finite dimensional parameter $\theta$ taking values in the known parameter space $\boldsymbol \Theta$. We postulate that all these models are possible descriptions of some reality, which is of interest to us, and that only one of the models, say the  one corresponding to $\theta^*\in\boldsymbol \Theta$, is the adequate, or true, description of this reality.

Motivated by discrete time robust stochastic control problems subject to model uncertainty (cf. \cite{BCCCJ2016}), we consider in the present paper discrete time, time-homogeneous Markov chain models only. Accordingly, we assume that the one-step transition kernel of the Markov chain model is parameterized by $\theta$. We postulate that the true parameter $\theta^*$ is not known, and the main goal is to derive a recursive (in time) construction of confidence regions for $\theta^*$. Needless to say, we are seeking a recursive construction of confidence regions for $\theta^*$ that satisfy desired properties; in particular, some asymptotic properties, as the time series of observations increases. Robust stochastic control problems provide primary motivation for the present work, but, clearly, potential applications of the results presented here are far reaching.

The recursive construction of confidence regions is needed not only for the purpose of speeding up the computation of the successive confidence regions, but, primarily, for the ability to apply the dynamic programming principle in the context of robust stochastic control methodology introduced in \cite{BCCCJ2016}.

There is a vast literature devoted to recursive computation, also known as on-line computation, of point estimators. It is fair to say though that, to the best of our knowledge, the literature regarding recursive  construction of confidence regions and their asymptotic analysis is very scarce. In fact, we were able to identify only two previous works,  \cite{Yin1989} and \cite{Yin1990}, touching upon this subject. In this regard, our work is the first to fully concentrate on the recursive construction of confidence regions and their asymptotic analysis. The geometric idea that underlies our recursive construction is motivated by recursive representation of confidence intervals for the mean of one dimensional Gaussian distribution with known variance, and by recursive representation of confidence ellipsoids  for the mean and variance of one dimensional Gaussian distribution, where in both cases observations are generated by i.i.d. random variables. The recursive representation is straightforward in the former case,  but it is not so any more in the latter one.

In the already mentioned works, \cite{Yin1989} and \cite{Yin1990}, a closely related idea was used for constructions of sequentially determined confidence ellipsoids based on stopping Brownian motions.
Our paper and Yin's work compare as follows:

  We take ergodic Markov chains as our underlying processes, whereas Yin considered several other different processes such as moving average processes and stationary $\phi$-mixing processes.  While providing a formula for the confidence ellipsoids that is essentially a recursive formula, Yin was interested in computing the volume of the ellipsoids.  By defining a stopping time as the first time that the volume of a ellipsoid is smaller than some threshold, in \cite{Yin1989} and \cite{Yin1990} the author proved a series of properties for such stopping time, and developed a stopping rule for recursive on-line algorithms.  We, on the other hand, focus on a constructive recursive derivation of the extreme points of our confidence regions.  Specifically, we provide a recursive formula for the extreme points so that we are able to efficiently compute these points and therefore to efficiently represent the points that lie in the confidence regions. This is an important new development, as it allows us to apply dynamic programming principle to the robust stochastic control problem that is studied in \cite{BCCCJ2016}.   From the numerical point of view, formulae for extreme points of the ellipsoids lead to efficient solution to the optimization problems that we encounter. We prove the weak consistency for the recursive confidence regions, for which having the representations of the extreme points plays an important role in the proof.

As it will be seen, one of the the key ingredients in our recursive construction of confidence regions is an appropriate  recursive scheme for deriving a point estimator of $\theta^*$. In this regard, building upon classical results from inferential statistics and from the area of stochastic approximation (cf. \cite{KushnerYin2003}, \cite{LeCam1956}, \cite{LeCam1960}), in Section 3:
\vspace{-0.3em}
\begin{itemize}
\item
We introduce the concept of \textit{quasi-asymptotic linearity of a point estimator} of $\theta^*$, which is satisfied by the recursive point estimation scheme that we develop in Section~\ref{section:ale}.
This concept is related to the classic definition of asymptotic linearity of a point estimator, but it overcomes one serious drawback that the classic concept suffers from: asymptotic linearity fails to be reconciled with the full recursiveness in some applications.
  \item Starting from what we call the base recursive point estimation scheme, we design a  quasi-asymptotically linear recursive point estimation scheme, and we prove the weak consistency and asymptotic normality of the point estimator  generated by this scheme.
\end{itemize}
\vspace{-0.2em}
The main original contribution of this paper is provided in Section 4 and it can be summarized as follows:
\begin{itemize}
  \item[] We provide a recursive construction of confidence regions for $\theta^*$. We prove that these confidence regions are weakly consistent, that is, they  converge in probability (in the Hausdorff metric) to the true parameter $\theta^*$.
\end{itemize}

The paper is organized as follows. In Section \ref{sec:pre}  we introduce the Markov chain framework relevant for the present study.

In Section \ref{sec:recpoint} we provide two recursive schemes for derivation of point estimators.
Section \ref{sec:sqrt} is devoted to the recursive construction of what we call the \textit{base (recursive) point estimator} of $\theta^*$. In our set-up, point-estimating of $\theta^*$ translates to finding solution to equation \eqref{eq:maineq}. This is an unknown equation. One of the most widely used iterative root finding procedures for unknown equations is the celebrated stochastic approximation method. Our base (recursive) point estimation scheme for $\theta^*$ is an adaptation of the  stochastic approximation method. Also, here we prove the strong consistency of the base point estimator. The key step to the desired recursive construction of confidence regions for $\theta^*$ is to establish the asymptotic normality of the underlying recursive point estimator.
In this regard one could impose additional assumptions, on top of the conditions needed for proving consistency of the point estimator, to obtain asymptotic normality for stochastic approximation estimators (see e.g. \cite{Sacks1958}, \cite{Fabian1968}, \cite{LaiRobbins1979}). This however would typically result in imposing a long list of assumptions that would not be easily verifiable.  Thus, we choose to go into a different direction by modifying the base estimator $\tilde \theta$ to the effect of producing a recursive estimator that is asymptotically normal. Therefore, in Section \ref{section:ale} we appropriately  modify our base (recursive) point estimator, so to construct a quasi-asymptotically linear (recursive) point estimator, for which we prove weak consistency and asymptotic normality.

The main section of this paper is Section~\ref{sec:conc}, which contains the main original contribution of the paper. This section is devoted to recursive construction of confidence regions for $\theta^*$, and to studying their asymptotic properties. In particular, we show that confidence regions derived from quasi-asymptotically linear (recursive) point estimators preserve a desired geometric structure. Such structure guarantees that we can represent the confidence regions in a recursive way in the sense that the region produced at step $n$ is fully determined by the region produced at step $n-1$ and by the newly arriving observation of the underlying reality.

Illustrating examples are provided in Section \ref{sec:ex}. The paper is completed with technical Appendices that also contains some of the proofs.

\section{Preliminaries}\label{sec:pre}

Let  $(\Omega, \sF)$ be a measurable space. The non-empty compact hyperrectangle $\boldsymbol \Theta\subset
\bR^d$  will play the role of the parameter
space throughout.\footnote{In general, the parameter space
may be infinite dimensional, consisting for example of dynamic
factors, such as deterministic functions of time or hidden Markov
chains. In this study, for simplicity, we chose the parameter
space to be a subset of $\bR^d$.}
We define $C(\theta)$, $\theta\in\partial\boldsymbol\Theta$, to be the infinite convex cone generated by the outer normals at $\theta$ of the faces on which $\theta$ lies; and $C(\theta)=\{0\}$ if $\theta$ belongs to the interior of $\boldsymbol\Theta$.
On the space $(\Omega, \sF)$ we consider a discrete time, real valued random process $Z=\{
Z_n,\ n\geq0\}$.\footnote{The study presented in this paper extends to the case when process $Z$ takes values in ${\mathbb{R}}^m$, for $m>1.$ We focus here the case of $m=1$ for simplicity of presentation.} We
postulate that this process is observed, and we denote by ${\mathbb {F}}=(\sF_n,n\geq0)$ its natural
filtration.
The (true) law of $Z$  is unknown, and assumed to belong
to a parameterized family of probability distributions on
$(\Omega, \sF)$, say $\{\bP_\theta,  \theta\in \boldsymbol \Theta\}$.  It will be convenient to consider $(\Omega, \sF)$ to be the canonical space for $Z$, and to consider $Z$ to be the canonical process (see Appendix A for details). Consequently, the law of $Z$ under $\bP_\theta$ is the same as $\bP_\theta$. The (true) law of $Z$ will be denoted by $\bP_{\theta^*}$; accordingly, $\theta ^*\in\boldsymbol \Theta$ is the (unknown) true parameter.
We assume that $\theta^*$ lies in the interior of $\boldsymbol \Theta$.

The set of probabilistic models that we are concerned with is $\set{(\Omega, \sF,{\mathbb {F}},Z,\bP_\theta),\theta\in \boldsymbol \Theta}.$ The model uncertainty addressed in this
paper occurs if  $\boldsymbol \Theta \ne \{\theta^*\},$ which we
assume to be the case. Our objective is to provide a recursive construction of confidence regions for $\theta^*$, based on accurate observations of realizations of process $Z$ through time, and satisfying desirable asymptotic properties.

In what follows, all equalities and inequalities between random variables will be understood in $\bP_{\theta^*}$ almost surely sense. We denote by $\bE_{\theta^*}$ the expectation operator corresponding to probability $\bP_{\theta^*}$.

We impose the following structural standing assumption.

\noindent \textbf{Assumption M:}\\ (i) Process $Z$ is a time homogenous Markov process under any $\bP_\theta, \ \theta \in \boldsymbol \Theta$.\\
(ii) Process $Z$ is an ergodic Markov process under $\bP_{\theta^*}$.\footnote{See Appendix \ref{sec:A1} for the definition of ergodicity that we postulate here.} \\
(iii) The transition kernel of process $Z$ under any $\bP_\theta, \ \theta \in \boldsymbol \Theta$ is absolutely continuous with respect to the Lebesgue measure on $\bR$, that is, for any Borel subset $A$ of $\bR$
\[
\bP_\theta (Z_1\in A \mid Z_0=x)=\int _A p_\theta(x,y) dy,
\]
for some positive and measurable function $p_\theta.$\footnote{This postulate is made solely in order to streamline the presentation. In general, our methodology works for Markov processes for which the transition kernel is not absolutely continuous with respect to the Lebesgue measure.}

For any $\theta \in \boldsymbol \Theta$ and $n\geq 1$, we define $\pi_n(\theta):=\log p_\theta(Z_{n-1},Z_n)$.

\begin{remark}\label{rem:law-of-pi}
 Since the process $Z$ is ergodic then is a stationary (see Appendix \ref{sec:A1}) process under $\bP_{\theta^*}$. Consequently, under $\bP_{\theta^*}$, for each $\theta \in \boldsymbol \Theta$ and for each $n\geq 0,$ the law of $\pi_n(\theta)$ is the same as the law of $\pi_1(\theta)$.
\end{remark}

We will need to impose several technical assumptions in what follows. We begin with the assumption

\medskip\noindent
R0. For any $\theta\in\boldsymbol \Theta$, $\pi_1(\theta)$ is integrable under $\bP_{\theta^*}$.

\medskip\noindent Then, assuming that M and R0 hold and using Proposition~\ref{th:marbirk} as well as the Kullback-Leibler Lemma (cf. \cite{KullbackLeibler1951}), we see that the following properties are satisfied:

    \noindent For any $\theta\in\boldsymbol \Theta$,
    \begin{equation}\label{eq:pi-ergo}
    \lim_{n\rightarrow\infty}\frac{1}{n}\sum_{i=1}^{n}\pi_i(\theta)=\bE_{\theta^*}[\pi_1(\theta)].
    \end{equation}

    \noindent Moreover, for any $\theta\in\boldsymbol \Theta$,
    \begin{equation}\label{eq:pi-ineq}
    \bE_{\theta^*}[\pi_1(\theta^*)]\geq \bE_{\theta^*}[\pi_1(\theta)].
    \end{equation}

In the statement of the technical assumptions R1-R8 below we use the notations
\begin{align}\label{notations}
  \psi_n(\theta)   &=\nabla\pi_n(\theta),\quad
  \Psi_n(\theta)   =\mH\pi_n(\theta),\quad   b_n(\theta)     =\bE_{\theta^*}[\psi_n(\theta)|\sF_{n-1}],
   \end{align}
where $\nabla$ denotes  the gradient vector and $\mH$ denotes the Hessian matrix with respect to $\theta$, respectively. Due to the fact that $Z$ is a Markov process, we have
\begin{eqnarray}\label{eq:z-markov}
b_n(\theta)=\bE_{\theta^*}[\psi_n(\theta)|\sF_{n-1}]=\bE_{\theta^*}[\psi_n(\theta)|Z_{n-1}],
\end{eqnarray}
which implies that for each $n\geq1$, $b_n$ is indeed a functional of $\theta$ and $Z_{n-1}$, and we postulate that $b_n$ is continuous with respect to $Z_{n-1}$.

\noindent
\begin{enumerate}[\textrm{R}1.]
\item
For each $x,y\in\bR$ the function $p_\cdot(x,y)\, : \boldsymbol \Theta \rightarrow \bR_+$ is three times differentiable, and
\begin{equation}\label{eq:HP}
  \nabla\int_\bR p_\theta(x,y)dy=\int_{\bR}\nabla p_{\theta}(x,y)dy, \quad \mH\int_\bR p_\theta(x,y)dy=\int_{\bR}\mH p_{\theta}(x,y)dy.
\end{equation}
\item
For any $\theta\in\mathbf{\Theta}$, $\psi_1(\theta)$ and $\Psi_1(\theta)$ are integrable under $\bP_{\theta^*}$.
The function $\bE_{\theta^*}[\pi_1(\, \cdot\,)]$ is twice differentiable in $\theta$, and
\begin{align*}
\nabla\bE_{\theta^*}[\pi_1(\theta)]=\bE_{\theta^*}[\psi_1(\theta)],\quad
\mH\bE_{\theta^*}[\pi_1(\theta)]=\bE_{\theta^*}[\Psi_1(\theta)].
\end{align*}
\item There is no stationary point\footnote{A stationary point of an ODE $\frac{dx}{dt}=f(x(t))$ is a point $x$ such that $f(x)=0$. For detailed discussion about projected ODE and stationary points, please refer to \cite{KushnerYin2003}.} on $\partial\boldsymbol\Theta$ for the differential equation
\begin{equation}\label{eq:proj-ode}
\frac{dx(t)}{dt}=\bE_{\theta^*}[\psi_1(x(t))]+\zeta(t),
\end{equation}
where $\zeta(t)\in-C(x(t))$ is the minimum force needed to keep $x(\cdot)$ in $\boldsymbol\Theta$. There exists a unique $\theta^\star\in\mathbf{\Theta}$ such that
$$
\bE_{\theta^*}[\psi_1(\theta^\star)]=0.
$$
  \item
  There exist some  positive constants $K_i, i=1,2,3$, such that for any $\theta,\theta_1,\theta_2\in\mathbf{\Theta}$, and $n\geq1$,\footnote{Superscript $T$ will  denote the transpose.}
  \begin{align}
    (\theta-\theta^*)^Tb_n(\theta)&\leq  - K_1\|\theta-\theta^*\|^2,\label{eq:K1} \\
    \|b_n(\theta_1)-b_n(\theta_2)\|&\leq K_2\|\theta_1-\theta_2\|, \label{eq:K2} \\
       \bE_{\theta^*}[\|\Psi_n(\theta_1)-\Psi_n(\theta_2)\|\mid \sF_{n-1}]&\leq K_3\|\theta_1-\theta_2\|.\label{eq:K3}
  \end{align}
  \item
  There exists a positive constant $K_4$, such that for any $\theta\in\mathbf{\Theta}$, and $n\geq1$,
  \begin{align}\label{eq:K4}
    \bE_{\theta^*}[\|\mH\psi_n(\theta)\||\sF_{n-1}]\leq K_4.
  \end{align}
  \item
  For any $n\geq1$,
  \begin{align}\label{eq:boundedE}
    \sup_{\theta \in \mathbf{\Theta}} \bE_{\theta^*}\|\psi_n(\theta)-b_n(\theta)\|^2<\infty.
  \end{align}

  \item
  For each $\theta\in\mathbf{\Theta}$ the Fisher information matrix \[I(\theta):={\bE_{\theta}[\psi_1(\theta)\psi_1^T(\theta)]}\] exists and is positive definite. Moreover, $I(\theta)$ is continuous with respect to $\theta$.
  \item
  \begin{align}
  \lim_{n\to\infty}\bE_{\theta^*}\left[\sup_{0\leq i\leq n}\left|\frac{1}{\sqrt{n}}\psi_i(\theta^*)\right|\right]=0.\label{eq:clt}
  \end{align}
\end{enumerate}

\begin{remark} (i) Note that in view of the Remark \ref{rem:law-of-pi} properties assumed in R2, R3, and R8 imply that analogous properties hold with time $n$ in place of time $1$.\\ (ii)
  According to Proposition~\ref{prop:mtheta}, we have that if R4-R6 hold, then \eqref{eq:K1}-\eqref{eq:K4} are also satisfied for any $\sF_{n-1}$-measurable random vector $\boldsymbol{\theta} \in \mathbf{\Theta}$.\\
  (iii)  A detailed analysis of the  ODE~\eqref{eq:proj-ode} is not given here due to space limitation. However, it proceeds in analogy to what is done in \cite[Section 4.3]{KushnerYin2003} . One concludes that this equation admits a unique solution for $x(0)\in\boldsymbol\Theta$.

\end{remark}

As stated above, our aim is to provide a recursive construction of the confidence regions for $\theta^*$.
  In the sequel, we will propose a method for achieving this goal that will be derived from a suitable recursive point estimator of $\theta^*$.
  Note that due to \eqref{eq:pi-ineq} and Assumption~R3, we have that $\theta^*$ is the unique solution of
  \begin{align}\label{eq:maineq}
  \bE_{\theta^*}[\psi_1(\theta)]=0.
  \end{align}
  Therefore, point-estimating $\theta^*$ is equivalent to point-estimating the solution of the equation \eqref{eq:maineq}. Since $\theta^*$ is unknown, the left-hand-side of the equation \eqref{eq:maineq} is not really known to us. We will therefore apply an appropriate version of the so called \textit{stochastic approximation} method, which is a recursive method used to point-estimate zeros of functions that can not be directly observed. This can be done in our set-up since, thanks to \eqref{eq:pi-ergo}, we are provided with a sequence of observed random variables  $\frac{1}{n}\sum_{i=1}^{n}\psi_i(\theta)$ that $\bP_{\theta^*}$ almost surely converges to $\bE_{\theta^*}[\psi_1(\theta)]$ -- a  property, which will enable us to adopt the method of stochastic approximation.
   Accordingly, in the next two sections, we will introduce two recursive point estimators of $\theta^*$, and we will derive  properties of these estimators that are relevant for us.

\section{Recursive point estimators }\label{sec:recpoint}
In this section two types of recursive point estimators are derived that are needed for construction of the recursive confidence regions from Section~\ref{sec:conc}.

\subsection{$\sqrt{n}$-consistent base point estimator}\label{sec:sqrt}

   In this section we consider a recursive point estimator $\tilde \theta=\set{\tilde{\theta}_n,\, n\geq 1}$ of $\theta^*$, that will be defined in \eqref{eq:tre}. Towards this end, we fix a  positive constant $\beta$  such that $\beta K_1>\frac{1}{2}$, where $K_1$ was introduced in Assumption~R6.
   Then, we follow the idea in \cite{KushnerYin2003} and define  the  process $\tilde{\theta}$ recursively as follows,
     \begin{align}
  \tilde{\theta}_n=\tilde{\theta}_{n-1}+\frac{\beta}{n}\psi_n(\tilde{\theta}_{n-1})+\frac{\beta}{n}J_n,\quad n\geq1, \label{eq:tre}
  \end{align}
  with the initial guess $ \tilde \theta _0$ being an element in $\mathbf{\Theta}$, where $\psi_n$ was defined in \eqref{notations}.
  The projection term $J_n$ is chosen so that $\frac{\beta}{n}J_n$ is the vector of shortest Euclidean length needed to take $\tilde{\theta}_{n-1}+\frac{\beta}{n}\psi_n(\tilde{\theta}_{n-1})$ back to the set $\boldsymbol\Theta$.
  It is not hard to see that $J_n\in-C(\tilde{\theta}_n)$.

   Given the definition of $\psi_n$, we see that $\tilde{\theta}_n$ is updated from $\tilde{\theta}_{n-1}$ based  on new observation $Z_n$ available at time $n$; of course, $Z_{n-1}$ is used as well. We note that the recursion \eqref{eq:tre} is a version of  the constrained stochastic approximation method,  which  is meant to recursively approximate roots of the unknown equations, such as equation \eqref{eq:maineq} (see e.g. \cite{RobbinsMonro1951}, \cite{KieferWolfowitz1952}, \cite{LjungSoederstroem1987}, \cite{KushnerClark1978}, \cite{KushnerYin2003}).

   \begin{remark}
   In applications of stochastic approximation, there are two ways to deal with the case of iterates becoming too large. One is to impose some stability conditions on the problem, and the other is to make appropriate adjustments to the basic algorithm.
   The latter is usually called constrained or truncated stochastic approximation (see e.g. \cite{KushnerClark1978}, \cite{KushnerBuche2002}, \cite{ShariaZhong2016}). In this work, we use the second method so that assumptions R1-R8 only need to be satisfied for $\theta$ that belongs to a compact subset of $\bR^d$ instead of the whole space.
   \end{remark}

   \begin{remark}
     In practice, for $\boldsymbol\Theta$ that is defined as a hyperrectangle, the projection term $J_n$ is easily computable.
     See \eqref{eq:j1} and \eqref{eq:j2} as an example in the two dimensional case.
     It is also worth noting, as discussed in \cite{KushnerYin2003}, there are other feasible construction of $\boldsymbol\Theta$.
   \end{remark}

As mentioned above, we are interested in the study of asymptotic properties of confidence regions that we will construct recursively in Section~\ref{sec:conc}. These asymptotic properties crucially depend on the asymptotic properties of our recursive (point) estimators. One of such required properties is asymptotic normality.
As discussed earlier, we will modify the base estimator $\tilde \theta$ to the effect of producing a recursive estimator that is asymptotically normal. In the next section we will construct such estimator, denoted there as $\hat \theta$, and we will study its asymptotic properties in the spirit of the method proposed by Fisher~\cite{Fisher1925}. Motivated by finding estimators that share the same asymptotic property as maximum likelihood estimators (MLEs), Fisher proposed in \cite{Fisher1925} that if an estimator is $\sqrt{n}$-consistent (see below), then appropriate modification of the estimator has the same asymptotic normality as the MLE. This subject was further studied by LeCam in \cite{LeCam1956} and \cite{LeCam1960}, where a more general class of observation than i.i.d. observations are considered.

Accordingly, we will show that $\tilde{\theta}$ is strongly consistent, and, moreover it maintains $\sqrt{n}$ convergence rate, i.e.
\begin{align}\label{eq:rootncon}
\bE_{\theta^*}\|\tilde{\theta}_n-\theta^*\|^2=O(n^{-1}).
\end{align}
An estimator that satisfies this equality is said to be $\sqrt{n}$-\textit{consistent}.

For convenience, throughout, we will use the notation $\Delta_n:=\tilde{\theta}_n-\theta^*$, $n\geq 1$.

Next two results show that $\tilde{\theta}$ is strongly consistent and $\sqrt{n}$-\textit{consistent}. The proofs of these results are deferred to the Appendix~\ref{sec:techSup}.

   \begin{proposition}\label{pr:consistency}
     Assume that R1-R3, and \eqref{eq:K2} are satisfied,
     then
     \begin{eqnarray}\label{eq:consistency}
     \lim_{n\to\infty}\tilde{\theta}_n=\theta^*, \quad \bP_{\theta^*}-a.s.
     \end{eqnarray}
   \end{proposition}

\begin{proposition}\label{th:rootncon}
  Assume that \eqref{eq:K1}, \eqref{eq:K2} and \eqref{eq:boundedE} hold. Then,
  $$
  \bE_{\theta^*}\|\tilde{\theta}_n-\theta^*\|^2=O(n^{-1}).
  $$
\end{proposition}

\subsection{Quasi-asymptotically linear estimator}\label{section:ale}

In this section we define a new estimator denoted as $\{\hat{\theta}_n, n\geq1\}$ and given recursively by
\begin{equation}\label{eq:hre}
  \begin{aligned}
  \hat{\theta}_n&=-I^{-1}(\tilde{\theta}_n)I_n\tilde{\theta}_n+I^{-1}(\tilde{\theta}_n)\Gamma_n,\\
\Gamma_n&=\frac{n-1}{n}\Gamma_{n-1}+\frac{1}{n}(\textrm{Id}+\beta I_n)\psi_n(\tilde{\theta}_{n-1})+\frac{\beta}{n}I_nJ_n,\\
I_n&=\frac{n-1}{n}I_{n-1}+\frac{1}{n}\Psi_n(\tilde{\theta}_{n-1}), \quad n\geq1,\\
\Gamma_0&=0, \quad I_0=0,
  \end{aligned}
  \end{equation}
where $\textrm{Id}$ is the unit matrix. Since $\tilde{\theta}_n$, $I_n$, and $\Gamma_n$ are updated from time $n-1$ based on the new observation $Z_n$ available at time $n$, then the estimator $\hat{\theta}$ indeed is recursive.
This estimator will be used in Section 6 for recursive construction of confidence regions for $\theta^*$.

\begin{remark} In the argument below we will use the following representations of $\Gamma_n$ and $I_n$,
\[
\Gamma_n=\frac{1}{n}\sum_{j=1}^n\left[\left(\textrm{Id}+\beta I_j\right)\psi_j(\tilde{\theta}_{j-1})+\beta I_jJ_j\right],\quad I_n=\frac{1}{n}\sum_{i=1}^{n}\Psi_i(\tilde{\theta}_{i-1}).
\]
\end{remark}

\noindent
Next, we will show that $\hat{\theta}$ is weakly consistent and asymptotically normal. We will derive asymptotic normality of  $\hat{\theta}$ from the property of quasi-asymptotic linearity, which is related to the property of asymptotic linearity (cf. \cite{Shiryayev1984}), and which is defined as follows:
\begin{definition}\label{def:GALE}
  An estimator $\{\bar{\theta}_n, n\geq1\}$  of $\theta^*$ is called a \textit{quasi-asymptotically linear estimator} if there exist a $\bP_{\theta^*}$-convergent, adapted matrix valued process $G$, and adapted vector valued processes $\vartheta$ and $\varepsilon$, such that
  $$
    \bar{\theta}_n-\vartheta_n=\frac
  {G_n}{n}\sum_{i=1}^n\psi_i(\theta^*)+\varepsilon_n,\ n\geq1,\quad  \vartheta_n\xrightarrow[n \to \infty ]{\bP_{\theta^*}}\theta^*,\quad \sqrt{n}\varepsilon_n\xrightarrow[n \to \infty]{\bP_{\theta^*}}0.$$
\end{definition}

 Our definition of quasi-asymptotically linear estimator is motivated by the classic concept of asymptotically linear estimator (see e.g. \cite{Sharia2010}):  $\check {\theta}$ is called (locally) asymptotically linear if there exists a matrix process $\{\check G_n,n\geq1\}$ such that
  \begin{align*}
    \check {\theta}_n-\theta^*=\check G_n\sum_{i=1}^{n}\psi_i(\theta^*)+\varepsilon_n,
  \end{align*}
  where $\check G_n^{-1/2}\varepsilon_n\xrightarrow[n \to \infty ]{\bP_{\theta^*}}0$.
Asymptotic linearity is frequently used in the proof of asymptotic normality of estimators.
However, in general, asymptotic linearity can not be reconciled with the full recursiveness of the point estimator.
The recursiveness of the point estimator is the key property involved in construction of recursive confidence regions. As it will be shown below, the fully recursive estimator $\hat \theta$ is quasi-asymptotically linear. 

In what follows, we will make use of the following representation for $\hat\theta$
\begin{equation}
  \label{eq:HthetaAlt}
   \hat{\theta}_n=-I^{-1}(\tilde{\theta}_n)I_n\tilde{\theta}_n+\frac{1}{n}I^{-1}(\tilde{\theta}_n)\sum_{j=1}^n\left[\left(\textrm{Id}+\beta I_j\right)\psi_j(\tilde{\theta}_{j-1})+\beta I_jJ_j\right].
\end{equation}

\begin{theorem}\label{th:mainth}
  Assume that R1--R8 hold, then the estimator $\hat{\theta}$ is $\bP_{\theta^*}$--weakly consistent.\footnote{That is, $
  \hat{\theta}_n\xrightarrow[n\to\infty]{\bP_{\theta^*}}\theta^*.
  $}\\
  Moreover, $\hat{\theta}$ is quasi-asymptotically linear estimator for $\theta^*$.
\end{theorem}
\noindent The proof is differed to the Appendix~\ref{sec:techSup}.

The next result, which will be used in analysis of asymptotic properties of the recursive confidence region for $\theta^*$ in Section 6, is an application of Theorem~\ref{th:mainth}.

\begin{proposition}\label{th:anormal}
  Assume that R1--R9 are satisfied. Then, there exists an adapted process $\vartheta$ such that
 \begin{equation}\label{1}
  \vartheta_n\xrightarrow[n\to\infty]{\bP_{\theta^*}}\theta^*,
 \end{equation}
  and
   \begin{equation}\label{2}
  \sqrt{n}(\hat{\theta}_n-\vartheta_n)\xrightarrow[n\to\infty]{d}\mathcal{N}(0,I^{-1}(\theta^*)).
  \end{equation}
\end{proposition}
\noindent See Appendix~\ref{sec:techSup} for the proof.

We end this section with the following technical result, which will be used in our construction of confidence region in Section 6. Towards this end, for any $\theta\in\mathbf{\Theta}$ and $n\geq 1$, we define\footnote{We use superscripts here to denote components of vectors and matrices.}
\begin{align}\label{eq:defUn}
U_n(\theta)&: =n(\hat{\theta}_n-\theta)^TI(\tilde{\theta}_n)(\hat{\theta}_n-\theta) \\
& \ =n\sum^d_{i=1}\sum^d_{j=1}\sigma^{ij}_n(\hat{\theta}^i_n-\theta^i)(\hat{\theta}^j_n-\theta^i), \nonumber
\end{align}
where $(\sigma_n^{ij})_{i,j=1,\ldots,d}=I(\tilde{\theta}_n)$, and, as usual, we denote by $\chi^2_d$ a random variable that has the chi-squared distribution with $d$ degrees of freedom.

\begin{corollary}\label{th:chisquare}
  With $\vartheta_n=-I^{-1}(\tilde{\theta}_n)I_n\theta^*$, we have that
    $$
  U_n(\vartheta_n)\xrightarrow[n\to\infty]{d}\chi^2_d.
  $$
\end{corollary}
\begin{proof}
  From Assumption~R8, strong consistency of $\tilde{\theta}$ and Proposition~\ref{th:anormal},  and employing the Slutsky's theorem again, we get that
  $$
  \sqrt{nI(\tilde{\theta}_n)}(\hat{\theta}_n-\vartheta_n)\xrightarrow[n\to\infty]{d}\cN(0,\textrm{Id}).
  $$
   Therefore,
  $$
  U_n(\vartheta_n)=n(\hat{\theta}_n-\vartheta_n)^TI(\tilde{\theta}_n)(\hat{\theta}_n-\vartheta_n)\xrightarrow{d} \xi^T \xi,
  $$
  where $\xi \sim \mathcal{N}(0,\textrm{Id}).$ The proof is thus complete since $\xi^T \xi \overset{d}{=} \chi^2_d.$

\end{proof}

\section{Recursive construction of confidence regions}\label{sec:conc}

This section is devoted to the construction of the recursive confidence region based on quasi-asymptotically linear estimator $\hat{\theta}$ developed in Section~\ref{section:ale}.
We start with introducing the definition of the approximated confidence region.
\begin{definition} Let $V_n:\bR^{n+1}\to2^{\mathbf{\Theta}}$ be a set valued function such that $V_n(z)$ is a connected set\footnote{A connected set is a set that cannot be represented as the union of two or more disjoint nonempty open subsets.} for any $z\in\mathbb{R}^{n+1}$.
The set $V_n(Z_0^n)$, with $Z_0^n:=(Z_0,\ldots,Z_n)$,  is called an \textit{approximated confidence region} for $\theta^*$, at significance level $\alpha\in(0,1)$, if there exists a weakly consistent estimator $\vartheta$ of $\theta^*$,  such that
  $$
  \lim_{n\to\infty}\bP_{\theta^*}(\vartheta_n\in V_n(Z_0^n))=1-\alpha.
  $$
\end{definition}

Such approximated confidence region can be constructed, as next result shows, by
using the asymptotic results obtained in Section~\ref{section:ale}.
Recall the notation $U_n(\theta) = n(\hat{\theta}_n-\theta)^TI(\tilde{\theta}_n)(\hat{\theta}_n-\theta)$, for $\theta\in\mathbf{\Theta}, \,n\geq1$.

\begin{proposition}
  Fix a confidence level $\alpha$, and let $\kappa\in\bR$ be such that $\bP_{\theta^*}(\chi_d^2<\kappa)=1-\alpha$.
  Then, the set
  $$
  \cT_n:=\{\theta\in\mathbf{\Theta}:U_n(\theta)<\kappa\}
  $$
  is an approximated confidence region for $\theta^*$.
\end{proposition}
\begin{proof}  As in Section~\ref{section:ale}, we take $\vartheta_n=-I^{-1}(\hat{\theta}_n)I_n\theta^*$, which in view of Proposition~\ref{th:anormal} is a weakly consistent estimator of $\theta^*$.
Note that $U_n(\, \cdot\, )$ is a continuous function, and thus $\cT_n$ is a connected set, for any $n\geq1$.
By Corollary~\ref{th:chisquare}, $U_n(\vartheta_n)\xrightarrow{d}\chi_d^2$, and since
  $
  \bP_{\theta^*}(\vartheta_n\in\cT_n)=\bP_{\theta^*}(U_n(\vartheta_n)<\kappa),
  $
we immediately have that $\lim_{n\to\infty}\bP_{\theta^*}(\vartheta_n\in\cT_n)=1-\alpha$.
This concludes the proof.
\end{proof}

Next, we will show that the approximated confidence region $\cT_n$ can be computed in a recursive way, by taking into account its geometric structure. By the definition, the set $\cT_n$ is the interior of a $d$-dimensional ellipsoid, and hence $cT_n$ is uniquely determined by its extreme $2d$ points. Thus, it is enough to establish a recursive formula for computing the extreme points.
Let us denote by
$$
(\theta_{n,k}^1,\ldots,\theta_{n,k}^d), \quad k=1,\ldots,2d,
$$
the coordinates of these extreme points; that is $\theta_{n,k}^i$, denotes the $i$th coordinate of the $k$th extreme point of ellipsoid $\cT_n$.

First, note that the matrix $I(\tilde{\theta}_n)$ is positive definite, and hence it admits the Cholesky decomposition:
  $$
  I(\tilde{\theta}_n)=L_nL^T_n=\begin{bmatrix}
  l^{11}_n & 0 & \cdots & 0 \\
  l^{21}_n & l^{22}_n & \cdots & 0 \\
  \vdots & \vdots & & \vdots \\
  l^{d1}_n & l^{d2}_n & \cdots & l^{dd}_n
  \end{bmatrix}
  \begin{bmatrix}
  l^{11}_n & l^{21}_n & \cdots & l^{d1}_n \\
  0 & l^{22}_n & \cdots & l^{d2}_n \\
  \vdots & \vdots & \cdots & \vdots \\
  0 & 0 & \cdots & l^{dd}_n
  \end{bmatrix},
  $$
  where $l^{ij}_n$ $i,j=1,\ldots,d$, are given by
  \begin{align*}
    l^{ii}_n&=\sqrt{\sigma^{ii}_n-\sum^{i-1}_{k=1}(l^{ik}_n)^2},\\
    l^{ij}_n&=\frac{1}{l^{ii}_n}\Big(\sigma^{ij}_n-\sum^{j-1}_{k=1}l^{ik}_nl^{jk}_n\Big).
  \end{align*}
  Thus, we have that $U_n(\theta)= n(u_{n,1}^2(\theta)+u_{n,2}^2(\theta)+\cdots+u_{n,d}^2(\theta))$, where
  \begin{align*}
    u_{n,i}(\theta)&=\sum_{j=i}^{d}l^{ji}_n(\hat{\theta}^j_n-\theta^j), \quad i=1,\ldots,d,
  \end{align*}
  and thus $\cT_n=\{\theta:\sum_{j=1}^d(u_{n,j}(\theta))^2<\frac{\kappa}{n}\}$.

By making the coordinate transformation $\theta\mapsto \rho$ given by $\rho=L_n^T(\hat{\theta}_n-\theta)$, the set $\cT_n$ in the new system of coordinates can be written as $\cT_n=\{\rho:\sum_{i=1}^{d}(\rho^i)^2<\frac{\kappa}{n}\}$.
Hence, $\cT_n$, in the new system of coordinates, is determined by the following $2d$ extreme points of the ellipsoid:
\begin{align*}
  (\rho^1_1,\ldots,\rho^d_1)&=(\sqrt{\frac{\kappa}{n}},0,\ldots,0),\\
   (\rho^1_2,\ldots,\rho^d_2)&=(-\sqrt{\frac{\kappa}{n}},0,\ldots,0),\\
   &\ldots\\
  (\rho^1_{2d-1},\ldots,\rho^d_{2d-1})&=(0,\ldots,0,\sqrt{\frac{\kappa}{n}}),\\
  (\rho^1_{2d},\ldots,\rho^d_{2d})&=(0,\ldots,0,-\sqrt{\frac{\kappa}{n}}).
\end{align*}
Then, in the original system of coordinates, the extreme points (written as vectors) are given by
  \begin{equation}\label{eq:extremepoints}
  \begin{aligned}
  (\theta_{n,2j-1}^1,\ldots,\theta_{n,2j-1}^d)^T&=\hat{\theta}_n-\sqrt{\frac{\kappa}{n}}(L_n^T)^{-1}e_j,\\
  (\theta_{n,2j}^1,\ldots,\theta_{n,2j}^d)^T&=\hat{\theta}_n+\sqrt{\frac{\kappa}{n}}(L_n^T)^{-1}e_j,
  \end{aligned}\ \qquad j=1,\ldots,d,
  \end{equation}
 where $\{e_j\}$, $j=1,\ldots, d$, is the standard basis in $\bR^d$.

Finally, taking into account the recursive constructions \eqref{eq:tre}, \eqref{eq:hre}, and the representation \eqref{eq:extremepoints}, we have the following recursive scheme for computing the approximate confidence region.

\noindent \textbf{Recursive construction of the confidence region}
    \begin{align*}
\quad\   \textrm{Initial Step:} \qquad&  \Gamma_0=0, \quad I_0=0, \quad \tilde{\theta}_0\in\mathbf{\Theta}. \\
   n^{\textrm{th}} \ \textrm{Step:}\qquad & \\
   \textrm{Input:} \quad & \tilde{\theta}_{n-1},I_{n-1},\Gamma_{n-1},Z_{n-1},Z_n.  \\
   \textrm{Output:} \quad &  \tilde{\theta}_n=\tilde{\theta}_{n-1}+\frac{\beta}{n}\psi_n(\tilde{\theta}_{n-1})+\frac{\beta}{n}J_n,\\
&I_n=\frac{n-1}{n}I_{n-1}+\frac{1}{n}\Psi_n(\tilde{\theta}_{n-1}),\\
&\Gamma_n=\frac{n-1}{n}\Gamma_{n-1}+\frac{1}{n}\left[(\textrm{Id}+\beta I_n)\psi_n(\tilde{\theta}_{n-1})+\beta I_nJ_n\right],\\
&(\theta_{n,2j}^1,\ldots,\theta_{n,2j}^d)^T
        =-I^{-1}(\tilde{\theta}_n)I_n\tilde{\theta}_n+I^{-1}(\tilde{\theta}_n)\Gamma_n+\sqrt{\frac{\kappa}{n}}(I_n^{-1/2})^Te_j,\\
&(\theta_{n,2j-1}^1,\ldots,\theta_{n,2j-1}^d)^T
        =-I^{-1}(\tilde{\theta}_n)I_n\tilde{\theta}_n+I^{-1}(\tilde{\theta}_n)\Gamma_n-\sqrt{\frac{\kappa}{n}}(I_n^{-1/2})^Te_j,\\
            & \qquad\qquad \qquad j=1,\ldots,d.
    \end{align*}
From here, we also conclude that there exists a function $\tau$, independent of $n$, such that
  \begin{equation}\label{nice}
  \cT_n=\tau(\cT_{n-1},Z_n).
  \end{equation}
The above recursive relationship goes to heart of application of recursive confidence regions in the robust adaptive control theory originated in \cite{BCCCJ2016}, since it makes it possible to take the full advantage of the dynamic programming principle in the context of such control problems.

We conclude this section by proving that the confidence region converges to the singleton $\theta^*$.
Equivalently, it is enough to prove that the extreme points converge to the true parameter $\theta^*$.

 \begin{proposition}
    For any $k\in\{1,\ldots,2d\}$, we have that
    $$
   \bP_{\theta^*}\!\textrm{-}\! \lim_{n\rightarrow\infty}\theta_{n,k}=\theta^*.
    $$
  \end{proposition}

  \begin{proof}

    By Assumption~R8 and Theorem~\ref{pr:consistency} (strong consistency of $\tilde{\theta}$), we have  that
    $
    L_n\xrightarrow[n\to\infty]{\textrm{a.s.}}I^{1/2}(\theta^*),
    $
    and consequently, we also have that
    \begin{equation}\label{eq:eL}
    \sqrt{\frac{\kappa}{n}}e^T_jL^{-1}_n\xrightarrow[n\to\infty]{\textrm{a.s}}0.
    \end{equation}
   Of course, the last convergence holds true in the weak sense too.
Passing to the limit in \eqref{eq:extremepoints}, in $\bP_{\theta^*}$~probability sense,  and using \eqref{eq:eL} and weak consistency of $\hat\theta$ (Theorem~\ref{th:mainth}), we finish the proof.

  \end{proof}

\section{Examples}\label{sec:ex}
In this section we will present three illustrative examples of the recursive construction of confidence regions developed above.
We start with our main example,  Example~\ref{ex:markov}, of a Markov chain with Gaussian transitional densities where both the conditional mean and conditional standard deviation are the parameters of interest. Example~\ref{ex:iid} is dedicated to the case of i.i.d. Gaussian observations, which is a particular case of the first example.

Generally speaking, the simple case of i.i.d. observations for which the MLE exists and asymptotic normality holds true, one can recursively represent the sequence of confidence intervals constructed in the usual (off-line) way, and the theory developed in this paper is not really needed. The idea is illustrated in Example~\ref{ex:mle} by considering again the same experiment as in Example~\ref{ex:iid}. In fact, as mentioned above, this idea served as the starting point for the general methodology presented in the paper.

\begin{example}\label{ex:markov}
  Let us consider a Markov process $\{Z_n\}$  with a Gaussian transition density function
  $$
  p_\theta(x,y)=\frac{1}{\sqrt{1-\rho^2}\sqrt{2\pi}\sigma}e^{-\frac{(y-\rho x-(1-\rho)\mu)^2}{2\sigma^2(1-\rho^2)}}, \quad n\geq1,
  $$
  and such that $Z_0\sim \cN(\mu,\sigma^2)$.

  We assume that the correlation parameter $\rho\in(-1,1)$ is known,  and the unknown parameter is $\theta=(\mu,\sigma^2)\in(-\infty,\infty)\times(0,\infty)$. The pair of true parameters $(\mu^*,(\sigma^*)^2)$ lies in the interior of $\boldsymbol\Theta=[a_1,a_2]\times[b_1,b_2]$, and $a_1\leq a_2, \ b_1\leq b_2$ are some fixed real numbers with $b_1>0$.

In the Appendix~\ref{sec:techSup} we show that the process $Z$ satisfies the Assumption M, and the conditions R0-R8.

Thus, all the results derived in the previous sections hold true. Moreover, for a given confidence level $\alpha$, we have the following explicit formulas for the $n$th step of the recurrent construction of the confidence regions:
\begin{align*}
  \tilde\mu_n&=\tilde\mu_{n-1}+\frac{\beta(Z_n-\rho Z_{n-1}-(1-\rho)\tilde\mu_{n-1})}{n\tilde\sigma_{n-1}^2(1+\rho)}+\frac{\beta}{n}J^1_n,\\
  \tilde\sigma_n^2&=\tilde\sigma_{n-1}^2-\frac{\beta}{n\tilde\sigma_{n-1}}+\frac{\beta(Z_n-\rho Z_{n-1}-(1-\rho)\tilde\mu_{n-1})^2}{n(1-\rho^2)\tilde\sigma_{n-1}^3})+\frac{\beta}{n}J^2_n,\\
  I_n&=\frac{n-1}{n}I_{n-1}+\frac{1}{n}\begin{bmatrix}
  -\frac{1-\rho}{(1+\rho)\tilde\sigma_{n-1}^2} & -\frac{2(Z_n-\rho Z_{n-1}-(1-\rho)\tilde\mu_{n-1})}{(1+\rho)\tilde\sigma_{n-1}^3}\\
  -\frac{2(Z_n-\rho Z_{n-1}-(1-\rho)\tilde\mu_{n-1})}{(1+\rho)\tilde\sigma_{n-1}^3} & \frac{1}{\tilde\sigma_{n-1}^2}-\frac{3(Z_n-\rho Z_{n-1}-(1-\rho)\tilde\mu_{n-1})^2}{(1-\rho^2)\tilde\sigma_{n-1}^4}
  \end{bmatrix},\\
  \Gamma_n&=\frac{n-1}{n}\Gamma_{n-1}+\frac{1}{n}(\textrm{Id}+\beta I_n)\begin{bmatrix}
  \frac{Z_n-\rho Z_{n-1}-(1-\rho)\tilde\mu_{n-1}}{\tilde\sigma_{n-1}^2(1+\rho)}\\
  -\frac{1}{\tilde\sigma_{n-1}}+\frac{(Z_n-\rho Z_{n-1}-(1-\rho)\tilde\mu_{n-1})^2}{(1-\rho^2)\tilde\sigma_{n-1}^3})
  \end{bmatrix}+\frac{\beta I_n}{n}\begin{bmatrix}
  J^1_n\\
  J^2_n
  \end{bmatrix},
  \end{align*}
and, for $j\in\set{1,2,3,4}$,
\begin{align*}
  \begin{bmatrix}
  \mu_{n,j}\\
  \sigma_{n,j}^2
  \end{bmatrix}&=-\begin{bmatrix}
  \frac{(1+\rho)\tilde{\sigma}^2_n}{1-\rho} & 0\\
  0 & \frac{\tilde\sigma^2_n}{2}
  \end{bmatrix}I_n\begin{bmatrix}
  \tilde\mu_{n}\\
  \tilde\sigma_{n}^2
  \end{bmatrix}+\begin{bmatrix}
  \frac{(1+\rho)\tilde{\sigma}^2_n}{1-\rho} & 0\\
  0 & \frac{\tilde\sigma^2_n}{2}
  \end{bmatrix}\Gamma_n + \varpi_j\frac{\kappa}{n}\begin{bmatrix}
  \sqrt{\frac{1+\rho}{1-\rho}}\tilde{\sigma}_n & 0\\
  0 & \frac{\tilde\sigma_n}{\sqrt{2}}
  \end{bmatrix}u_j,
\end{align*}
where  $\varpi_1=\varpi_3=-1$, $\varpi_2=\varpi_4=1$, $u_1=u_2=e_1$, $u_3=u_4=e_2$, $\beta$ is a constant such that $\beta>\frac{b_2^3}{4b_1}$, $\beta>\frac{(1+\rho)b_2^3}{2(1-\rho)b_1}$, and $\bP_{\theta^*}(\chi_2^2<\kappa)=1-\alpha$. The projection terms $J^1_n$ and $J^2_n$ are defined as follows,
\begin{equation}\label{eq:j1}
J^1_n=\left\{
\begin{tabular}{cl}
  $\frac{n}{\beta}(a_1-\accentset{\circ}{\mu}_n)$, \quad & $a_1>\accentset{\circ}{\mu}_n$, \\
  $\frac{n}{\beta}(a_2-\accentset{\circ}{\mu}_n)$, \quad & $a_2<\accentset{\circ}{\mu}_n$, \\
  0, \quad & otherwise, \\
\end{tabular}
\right.
\end{equation}
and
\begin{equation}\label{eq:j2}
J^2_n=\left\{
\begin{tabular}{cl}
  $\frac{n}{\beta}(b_1-\accentset{\circ}{\sigma}^2_n)$, \quad & $b_1>\accentset{\circ}{\sigma}^2_n$, \\
  $\frac{n}{\beta}(b_2-\accentset{\circ}{\sigma}^2_n)$, \quad & $b_2<\accentset{\circ}{\sigma}^2_n$, \\
  0, \quad & otherwise, \\
\end{tabular}
\right.
\end{equation}
where $\accentset{\circ}{\mu}_n=\tilde\mu_{n-1}+\frac{\beta(Z_n-\rho Z_{n-1}-(1-\rho)\tilde\mu_{n-1})}{n\tilde\sigma_{n-1}^2(1+\rho)}$, and $\accentset{\circ}{\sigma}^2_n=\tilde\sigma_{n-1}^2-\frac{\beta}{n\tilde\sigma_{n-1}}+\frac{\beta(Z_n-\rho Z_{n-1}-(1-\rho)\tilde\mu_{n-1})^2}{n(1-\rho^2)\tilde\sigma_{n-1}^3}$.

\end{example}

\begin{example}\label{ex:iid} Let $Z_n,\ n\geq0$, be a sequence of i.i.d. Gaussian random variables with an unknown mean $\mu$ and unknown standard deviation $\sigma$. Clearly, this important case is a particular case of Example~\ref{ex:markov}, with $\rho=0$, and the same recursive formulas for confidence regions by taking $\rho=0$ in the above formulas.
\end{example}

\begin{example}\label{ex:mle}
We take the same setup as in the previous example - i.i.d Gaussian random variables with unknown mean and standard deviation.
We will use the fact that in this case, the MLE estimators for $\mu$ and $\sigma^2$ are computed explicitly and given by
  $$
    \hat{\mu}_{n}=\frac{1}{n+1}\sum_{i=0}^{n}Z_i,\quad
    \hat{\sigma}^2_{n}=\frac{1}{n+1}\sum_{i=0}^{n}(Z_i-\hat{\mu}_n)^2, \quad n\geq1,
  $$
It is well known that $(\hat{\mu},\hat{\sigma}^2)$ are asymptotically  normal, namely
  \begin{align*}
    \sqrt{n}(\hat{\mu}_n-\mu^*,\hat{\sigma}^2_n-(\sigma^{*})^2)\xrightarrow[n\to\infty]{d}\cN(0,I^{-1}),
  \end{align*}
  where
  \begin{align*}
    I=\begin{bmatrix}
  (\sigma^{*})^2 & 0\\
  0 & 2(\sigma^{*})^4
  \end{bmatrix}.
  \end{align*}
First, note that $(\hat{\mu}_n,\hat{\sigma}^2_n)$ satisfies the following recursion:
\begin{equation}\label{eq:recMLE}
  \begin{aligned}
    \hat{\mu}_{n}&=\frac{n}{n+1}\hat{\mu}_{n-1}+\frac{1}{n+1}Z_{n},\\
    \hat{\sigma}^2_{n}&=\frac{n}{n+1}\hat{\sigma}^2_{n-1}+\frac{n}{(n+1)^2}(\hat{\mu}_n-Z_{n})^2,\quad n\geq1.
  \end{aligned}
\end{equation}
Second, due to asymptotic normality,   we also have that,  $U_n\xrightarrow[n\to\infty]{d}\chi_2^2$, where
$U_n:=\frac{n}{\hat{\sigma}^2_n}(\hat{\mu}_n-\mu^*)^2+\frac{n}{2\hat{\sigma}^{4}_n}(\hat{\sigma}^2_n-(\sigma^{*})^2)^2$.
Now, for a given confidence level $\alpha$, we let $\kappa\in\bR$ be such that $\bP_{\theta^*}(\chi^2_2<\kappa)=1-\alpha$, and
then, the confidence region for $(\mu,\sigma^2)$ is given by
  $$
  \cT_n:=\left\{(\mu,\sigma^2)\in\bR^2:\frac{n}{\hat{\sigma}^2_n}(\hat{\mu}_n-\mu)^2+\frac{n}{2\hat{\sigma}^{4}_n}(\hat{\sigma}^2_n-\sigma^{2})^2<\kappa\right\}.
  $$
Similar to the previous cases, we note that $\cT_n$ is the interior of an ellipse (in $\bR^2$), that is uniquely determined by its extreme points
\begin{align*}
&    (\mu_{n,1},\sigma^2_{n,1})=\left(\hat{\mu}_n+\sqrt{\frac{\kappa}{n}}\hat{\sigma}_n,\hat{\sigma}_n^2\right),
&&    (\mu_{n,2},\sigma^2_{n,2})=\left(\hat{\mu}_n-\sqrt{\frac{\kappa}{n}}\hat{\sigma}_n,\hat{\sigma}_n^2\right),\\
&    (\mu_{n,3},\sigma^2_{n,3})=\left(\hat{\mu}_n,\left(1+\sqrt{\frac{2\kappa}{n}}\right)\hat{\sigma}^2_n\right),
&&(\mu_{n,4},\sigma^2_{n,4})=\left(\hat{\mu}_n,\left(1-\sqrt{\frac{2\kappa}{n}}\right)\hat{\sigma}^2_n\right).
\end{align*}
  Therefore, taking into account \eqref{eq:recMLE}, we have a recursive formula for computing these extreme points, and thus the desired recursive construction of the confidence regions $\cT_n$.
\end{example}

%
%
%
%
%

\begin{appendix}
\section{Appendix}
\subsection{Ergodic Markov Chains}\label{sec:A1}
In this section, we will briefly recall some facts from the ergodic theory of Markov processes in discrete time.
Let $X$ be a time homogeneous Markov chain on a probability space $(\Omega, \cF,\bP)$, which takes values in a measurable space $(\cX,\mathfrak{X})$.
We refer to \cite[Chapter~4, Definition~2.6]{Revuz1984} for the definition of ergodicity for Markov processes.
If $X$ is an ergodic process under $\bP$, then it is also a stationary process, i.e. for any $n\geq1$, the law of $(X_j,X_{j+1},\ldots,X_{j+n})$ under $\bP$ is independent of $j$, $j\geq0$.

As usual, we denote by $\bE_{\bP}$ the expectation under $\bP$.
In view of  the classical Birkhoff's Ergodic Theorem, we have the following result, that will be used in this paper.
  \begin{proposition}\label{th:marbirk}
    Let $X$ be ergodic. Then for any $g$ such that $\bE_{\bP}[g(X_0,\ldots,X_n)]<\infty$, we have
    $$
    \lim_{N\to\infty}\frac{1}{N}\sum_{i=0}^{N-1}g(X_i,\ldots,X_{i+n})=\bE_{\bP}[g(X_0,\ldots,X_n)] \quad \bP-a.s.
    $$
  \end{proposition}
Next, we provide a brief discussion regarding sufficient conditions for the Markov chain $X$ to be ergodic.
Let $Q:\cX\times\mathfrak{X}\to[0,1]$ be the transition kernel of $X$. A probability measure $\pi$ on $(\cX,\mathfrak{X})$ is called an invariant measure of $Q$ if
  $$
  \int Q(x,A)d\pi(x)=\pi(A).
  $$
Let $\bP_\pi$ be the probability measure on $(\Omega,\cF)$ that is induced by $\pi$.
  \begin{proposition}\label{th:uniinv}
    If a transition kernel $Q$ has a unique invariant probability measure $\pi$, then $X$ is ergodic under $\bP_\pi$.
  \end{proposition}

  One powerful tool for checking the uniqueness of invariant probability measure is the notion of positive Harris chain. There are several equivalent definitions of positive Harris Markov chain, and we will use the one from \cite{HernandezLasserre2000}.

  \begin{definition}\label{def:pHarris}
    The Markov chain $X$ with transition kernel $Q$ is called a positive Harris chain if
    \begin{enumerate}[(a)]
      \item there exists a $\sigma$-finite measure $\mu$ on $\mathfrak{X}$ such that for any $x_0\in\cX$, and $B\in\mathfrak{X}$ with $\mu(B)>0 $
          $$
          \bP(X_n\in B \text{ for some } n<\infty|X_0=x_0)=1,
          $$
      \item there exists an invariant probability measure for $Q$.
    \end{enumerate}
  \end{definition}
\begin{remark}\label{prop:HarrisErgodic}
It is well known (cf. e.g. \cite{MeynTweedie1993}) that a positive Harris chain admits a unique invariant measure.
Thus, in view of Proposition~\ref{th:uniinv},  a positive Harris chain is also ergodic.
  \end{remark}

\subsection{CLT for Multivariate Martingales}
In this section, for a  matrix $A$ with real valued entries we denote by $|A|$ the sum of the absolute values of its entries.

In \cite{CrimaldiPratelli2005} Proposition~3.1, the authors gave the following version of the central limit theorem for discrete time multivariate martingales.
\begin{proposition}\label{CP}
  On a probability space $(\Omega,\sF,\bP)$ let $D=\{D_{n,j},0\leq j\leq k_n,n\geq1\}$ be a triangular array of $d$-dimensional real random vectors, such that, for each $n$, the finite sequence $\{D_{n,j},\1\leq j\leq k_n\}$ is a martingale difference process with respect to some filtration $\{\sF_{n,j},j\geq0\}$.
  Set
  $$
  D_n^*=\sup_{1\leq j\leq k_n}|D_{n,j}|, \quad U_n=\sum_{j=1}^{k_n}D_{n,j}D^T_{n,j}.
  $$
  Also denote by $\sU$ the $\sigma$-algebra generated by $\bigcup_j\sH_j$ where $\sH_j:=\liminf_n\sF_{n,j}$.
  Suppose that $D_n^*$ converges in $L^1$ to zero and that $U_n$ converges in probability to a $\sU$ measurable $d$-dimensional, positive semi-definite matrix $U$.   Then, the random vector $\sum_{j=1}^{k_n}D_{n,j}$ converges $\sU$-stably to the Gaussian kernel $\cN(0,U)$.
\end{proposition}

 \begin{remark}\label{CP1}
 $\sU$-stable convergence   implies convergence in distribution; it is enough to take the entire $\Omega$ in the definition of $\sU$-stable convergence. See for example \cite{AoldousEagelson1978} or \cite{HauslerLuschgy2015}.
 \end{remark}

We will apply the above proposition to the process $\{\psi_n(\theta^*),n\geq0\}$ such that Assumption M, R8 and R9 are satisfied. To this end, let us define the triangular array $\{D_{n,j},1\leq j\leq n,n\geq1\}$ as
$$
D_{n,j}=\frac{1}{\sqrt{n}}\psi_j(\theta^*),
$$
and let us take $\sF_{n,j}=\sF_j$.

First, note that $\bE_{\theta^*}[\psi_j(\theta^*)|\sF_{j-1}]=0$, so that for any $n\geq1$, $\{D_{n,j},1\leq j\leq n\}$ is a martingale difference process with respect to $\{\sF_j,0\leq j\leq n\}$.
Next, R9 implies that $D^*_n:=\sup_{1\leq j\leq n}\frac{1}{\sqrt{n}}|\psi_j(\theta^*)|$ converges in $L^1$ to 0.
Finally, stationarity, R8 and ergodicity guarantee that
$$
U_n:=\frac{1}{n}\sum_{j=1}^{n}\psi_j(\theta^*)\psi_j^T(\theta^*)\to\bE_{\theta^*}[\psi_1(\theta^*)\psi_1^T(\theta^*)] \quad \bP_{\theta^*}-a.s.
$$
The limit $I(\theta^*)=\bE_{\theta^*}[\psi_1(\theta^*)\psi_1^T(\theta^*)]$ is positive semi-definite, and it is deterministic, so that it is measurable with respect to any $\sigma$-algebra. Therefore, applying Proposition \ref{CP} and Remark \ref{CP1} we obtain
\begin{proposition}\label{prop:clt}
  Assume that Assumption M, R8, and R9 are satisfied. Then,
  $$
  \frac{1}{\sqrt{n}}\sum_{j=1}^{n}\psi_j(\theta^*)\xrightarrow[n\to\infty]{d}\cN(0,I(\theta^*)).
  $$
\end{proposition}

\subsection{Technical Supplement}\label{sec:techSup}
Assumptions~R4--R6 are stated for any deterministic vector $\theta\in\mathbf{\Theta}$. In this section, we show that if \eqref{eq:K1}-\eqref{eq:K4} hold for $\theta\in\mathbf{\Theta}$, then for any random vectors $\boldsymbol{\theta},\boldsymbol{\theta}_1,\boldsymbol{\theta}_2$ that are $\sF_{n-1}$ measurable and take values in $\mathbf{\Theta}$, analogous inequalities are true.

\begin{proposition}\label{prop:mtheta}
  Assume that R4-R6 are satisfied. Then, for any fixed $n\geq1$ and for any random vectors $\boldsymbol{\theta},\boldsymbol{\theta}_1,\boldsymbol{\theta}_2$ that are $\sF_{n-1}$ measurable and take values in $\mathbf{\Theta}$, we have
  \begin{align}
    (\boldsymbol{\theta}-\theta^*)^Tb_n(\boldsymbol{\theta})&\leq  - K_1\|\boldsymbol{\theta}-\theta^*\|^2,\label{eq:K1p} \\
    \|b_n(\boldsymbol{\theta})\|&\leq K_2\|\boldsymbol{\theta}-\theta^*\|, \label{eq:K2p} \\
       \bE_{\theta^*}[\|\Psi_n(\boldsymbol{\theta}_1)-\Psi_n(\boldsymbol{\theta}_2)\||\sF_{n-1}]&\leq K_3\|\boldsymbol{\theta}_1-\boldsymbol{\theta}_2\|,\label{eq:K3p}\\
    \bE_{\theta^*}[\|\mH\psi_n(\boldsymbol{\theta})\||\sF_{n-1}]&\leq K_4.\label{eq:K4p}
  \end{align}
\end{proposition}

\begin{proof}
  We will only show that \eqref{eq:K1p} is true. The validity of the remaining inequalities can be proved similarly.
  Also, without loss of generality, we assume that $d=1$.

  From \eqref{eq:K1}, we have for any $\theta\in\mathbf{\Theta}$,
  $
  (\theta-\theta^*)\bE_{\theta^*}[\psi_n(\theta)\mid\sF_{n-1}]\leq K_1|\theta-\theta^*|.
  $
  If $\boldsymbol{\theta}$ is a simple random variable, i.e. there exists a partition $\{A_m,1\leq m\leq M\}$ of $\Omega$, where $M$ is a fixed integer, such that $A_m\in\cF_{n-1}$, $1\leq m\leq M$, and $\boldsymbol{\theta}=\sum_{m=1}^{M}c_m\1_{A_m}$, where $c_m\in\mathbf{\Theta}$. Then, we have that
  \begin{align*}
    \left(\boldsymbol{\theta}-\theta^*\right)b_n(\boldsymbol{\theta})&
    =(\sum_{m=1}^{M}c_m\1_{A_m}-\theta^*)\bE_{\theta^*}[\psi_n(\boldsymbol{\theta})\mid\sF_{n-1}]\\
    &=\sum_{m=1}^{M}\1_{A_m}(c_m-\theta^*)\bE_{\theta^*}[\1_{A_m}\psi_n(\boldsymbol{\theta})\mid\sF_{n-1}]\\
    &=\sum_{m=1}^{M}\1_{A_m}(c_m-\theta^*)\bE_{\theta^*}[\1_{A_m}\psi_n(c_m)\mid\sF_{n-1}]\\
    &=\sum_{m=1}^{M}\1_{A_m}(c_m-\theta^*)\bE_{\theta^*}[\psi_n(c_m)\mid\sF_{n-1}]\\
    &\leq-\sum_{m=1}^{M}\1_{A_m}K_1|c_m-\theta^*|^2
    =-\sum_{m=1}^{M}K_1|\boldsymbol{\theta}-\theta^*|^2.
  \end{align*}
  From here, using the usual limiting argument we conclude that \eqref{eq:K1p} holds true for any $\sF_{n-1}$ measurable random variable $\boldsymbol{\theta}$.
\end{proof}

In the rest of this section we will verify that the Assumption~M and the properties R0--R8 are satisfies in Example~\ref{ex:markov}.

It is clear that the Markov chain $\{Z_n,n\geq0\}$, as defined in Example~\ref{ex:markov}, satisfies (i) and (iii) in Assumption~M. Next we will show that $Z$ is a  positive Harris chain (see Definition~\ref{def:pHarris}). For any Borel set $B\in\cB(\bR)$ with strictly positive Lebesgue measure, and any $z_0\in\bR$, we have that
\begin{align*}
  \lim_{n\to\infty}&\bP_{\theta^*}(Z_n\notin B,\ldots,Z_1\notin B\mid Z_0=z_0)\\
  &=\lim_{n\to\infty}\bP_{\theta^*}(Z_n\notin B\mid Z_{n-1}\notin B)\cdots\bP_{\theta^*}(Z_2\notin B\mid Z_1\notin B)\bP_{\theta^*}(Z_1\notin B\mid Z_0=z_0)\\
  &=\lim_{n\to\infty}\bP_{\theta^*}(Z_2\notin B\mid Z_1\notin B)^{n-1}\bP_{\theta^*}(Z_1\notin B\mid Z_0=z_0)=0,
\end{align*}
and thus $Z$ satisfies Definition~\ref{def:pHarris}.(a). Also, since the density (with respect to the Lebesgue measure) of $Z_1$ is
$$
f_{Z_1,\theta^*}(z_1)=\int_{\bR}p_{\theta^*}(z_0,z_1)f_{Z_0,\theta^*}(z_0)dz_0 = \frac{1}{\sqrt{2\pi}\sigma^*}e^{-\frac{(z_1-\mu^*)^2}{2(\sigma^*)^2}},
$$
then $Z_1\sim \cN(\mu^*,(\sigma^*)^2)$, and consequently, we get that  $Z_n\sim \cN(\mu^*,(\sigma^*)^2)$ for any $n\geq0$.
This implies that $\cN(\mu^*,(\sigma^*)^2)$ is an invariant distribution for $Z$.
Thus, $Z$ is a positive Harris chain, and respectively, by Remark~\ref{prop:HarrisErgodic}, $Z$ is an ergodic process.

As far as propreties R0--R8, we fist note that
  \begin{align*}
    \psi_n(\theta)&=\nabla\log p_\theta(Z_{n-1},Z_n)\\
    &=\Big(\frac{Z_n-\rho Z_{n-1}-(1-\rho)\mu}{\sigma^2(1+\rho)},-\frac{1}{\sigma}+\frac{(Z_n-\rho Z_{n-1}-(1-\rho)\mu)^2}{(1-\rho^2)\sigma^3}\Big)^T,\\
    b_n(\theta)&=\bE_{\theta^*}[\psi_n(\theta)|\sF_{n-1}]\\
    &=\Big(-\frac{(1-\rho)(\mu-\mu^*)}{\sigma^2(1+\rho)}, \frac{\sigma^{*,2}-\sigma^2}{\sigma^3}+\frac{(1-\rho)(\mu-\mu^*)^2}{(1+\rho)\sigma^3}\Big)^T,\\
  \Psi_n(\theta)&=\begin{bmatrix}
  -\frac{1-\rho}{(1+\rho)\sigma^2} & -\frac{2(Z_n-\rho Z_{n-1}-(1-\rho)\mu)}{(1+\rho)\sigma^3}\\
  -\frac{2(Z_n-\rho Z_{n-1}-(1-\rho)\mu)}{(1+\rho)\sigma^3} & \frac{1}{\sigma^2}-\frac{3(Z_n-\rho Z_{n-1}-(1-\rho)\mu)^2}{(1-\rho^2)\sigma^4}
  \end{bmatrix}.
  \end{align*}
We  denote by $Y_n:=Z_n-\rho Z_{n-1}-(1-\rho)\mu$, and we immediately deduce that that
  \begin{equation}\label{eq:yexp}
  \begin{aligned}
  \bE_{\theta^*}[Y_n\mid\sF_{n-1}]&=(1-\rho)(\mu^*-\mu), \\ \bE_{\theta^*}[Y^2_n\mid\sF_{n-1}]&=(1-\rho)^2(\mu-\mu^*)^2+(\sigma^{*})^2(1-\rho^2),\\
  \bE_{\theta^*}[Y^4_n\mid\sF_{n-1}]&=(1-\rho)^4(\mu^*-\mu)^4+6(1+\rho)(1-\rho)^3(\mu^*-\mu)^2(\sigma^{*})^2\\
  &\quad+3(\sigma^{*})^4(1-\rho^2)^2.
  \end{aligned}
  \end{equation}
  From here, and using the fact that $\mathbf{\Theta}$ is bounded, it is straightforward, but tedious,\footnote{The interested reader can contact the authors for details.} to show that R4, R5, and R6 are satisfied. Also, it is clear that R0 is true, and using \eqref{eq:yexp} by direct computations we get that R1 and R2 are satisfied.
Again by direct evaluations, we have that
$$
I(\theta)=\bE_{\theta}[\psi_1(\theta)\psi_1(\theta)^T]=\begin{bmatrix}
  \frac{1-\rho}{(1+\rho)\sigma^2} & 0\\
  0 & \frac{2}{\sigma^2}
  \end{bmatrix},
$$
which is positive definite matrix, and thus R7 is satisfied.

Since
  \begin{equation}\label{eq:ethetapsi}
  \bE_{\theta^*}[\psi_1(\theta)]=\left(\frac{(1-\rho)(\mu^*-\mu)}{\sigma^2(1+\rho)},
  \frac{(\sigma^{*})^2-\sigma^2}{\sigma^3}+\frac{(1-\rho)(\mu-\mu^*)^2}{(1+\rho)\sigma^3}\right),
  \end{equation}
then $\theta^*=(\mu^*,(\sigma^*)^2)$ is clearly the unique point that \eqref{eq:ethetapsi} is equal to 0.
Next we show that there is no $\theta\in\partial\boldsymbol\Theta$ such that $\bE_{\theta^*}[\psi_1(\theta)]+\zeta(\theta)=0$, where $\zeta(\theta)=(\zeta_1(\theta),\zeta_2(\theta))$ is defined in R3.
Towards this end, we assume the existence of $\theta^0=(\mu^0,(\sigma^0)^2)\in\partial\boldsymbol\Theta$ such that $\bE_{\theta^*}[\psi_1(\theta^0)]+\zeta(\theta^0)=0$.
Note that $\frac{(1-\rho)(\mu^*-\mu)}{\sigma^2(1+\rho)}<0$, if $\mu>\mu^*$; and $\frac{(1-\rho)(\mu^*-\mu)}{\sigma^2(1+\rho)}>0$, if $\mu<\mu^*$.
Hence, $\zeta_1(\theta)=0$ for any $\theta$, which implies that
\begin{equation}\label{eq:mu0}
\mu^0=\mu^*,
\end{equation}
and $\frac{(\sigma^{*})^2-(\sigma^0)^2}{(\sigma^0)^3}+\zeta_2(\theta^0)=0$.
Therefore, we have that
\begin{equation}\label{eq:zeta20}
\zeta_2(\theta^0)=\frac{(\sigma^0)^2-(\sigma^{*})^2}{(\sigma^0)^3}.
\end{equation}
The fact that $\theta^0\in\partial\boldsymbol\Theta$ and $\mu^0=\mu^*$ will imply $(\sigma^0)^2=b_1$, or $(\sigma^0)^2=b_2$.
This, together with \eqref{eq:zeta20}, yields
$$
\zeta_2(\theta^0)=\frac{b_1^2-(\sigma^{*})^2}{(\sigma^0)^3}, \quad \text{when } (\sigma^0)^2=b_1,
$$
or
$$
\zeta_2(\theta^0)=\frac{b_2^2-(\sigma^{*})^2}{(\sigma^0)^3}, \quad \text{when } (\sigma^0)^2=b_2,
$$
both of which cannot be true as we can easily check that $\zeta(\theta^0)\notin-C(\theta^0)$.
Thus, by contradiction we get that $\theta^0$ does not exist. So we conclude that there is no stationary point on $\partial\boldsymbol\Theta$ and R3 is satisfied.

Finally, we will verify R8. By Jensen's inequality and Cauchy-Schwartz inequality, we have that
\begin{align*}
\exp\left(\bE_{\theta^*}\sup_{0\leq i\leq n}|\psi_i(\theta^*)|\right)
& \leq\bE_{\theta^*}\exp\left(\sup_{0\leq i\leq n}|\psi_i(\theta^*)|\right)
  =\bE_{\theta^*}\left[\sup_{0\leq i\leq n}\exp|\psi_i(\theta^*)|\right]\\
&\leq \sum_{i=1}^{n}\bE_{\theta^*}\exp|\psi_i(\theta^*)|
\leq\sum_{i=1}^{n}\bE_{\theta^*}\exp(\frac{|Y_i|}{\sigma^2(1+\rho)}+ \frac{1}{\sigma}+\frac{Y_n^2}{(1-\rho)^2\sigma^3}) \\
&\leq\sum_{i=1}^{n}\Big(\bE_{\theta^*}\exp(\frac{2|Y_i|}{\sigma^2(1+\rho)})\Big)^{\frac{1}{2}}
\Big(\bE_{\theta^*}\exp(\frac{2}{\sigma}+\frac{2Y_i^2}{(1-\rho)^2\sigma^3})\Big)^{\frac{1}{2}}.
\end{align*}
Note that for $Y_i, i=0,\ldots,n$ is normally distributed, and therefore, there exist two constants $C_1$ and $C_2$, that depend on $\theta^*$ such that
$$
\bE_{\theta^*}\exp\left(\frac{2|Y_i|}{\sigma^2(1+\rho)}\right)=C_1, \quad
\bE_{\theta^*}\exp\left(\frac{2}{\sigma}+\frac{2Y_i^2}{(1-\rho)^2\sigma^3}\right)=C_2.
$$
Hence, we have that
$$
\bE_{\theta^*}\sup_{0\leq i\leq n}|\psi_i(\theta^*)|\leq\log n+\frac{1}{2}\log C_1C_2,
$$
and, thus  R8 is satisfied:
$$
\lim_{n\to\infty}\bE_{\theta^*}\Big[\sup_{0\leq i\leq n}\Big|\frac{1}{\sqrt{n}}\psi_i(\theta^*)\Big|\Big]
\leq\lim_{n\to\infty}\left(\frac{\log n}{\sqrt{n}}+\frac{\log C_1C_2}{2\sqrt{n}}\right)=0.
$$

\bigskip\noindent
\textbf{Proof of Proposition~\ref{pr:consistency}.} \\
     We will use Theorem~6.1.1 in \cite{KushnerYin2003} to show \eqref{eq:consistency}.
     We write our estimator in the following form
     $$
     \tilde{\theta}_n=\tilde{\theta}_{n-1}+\frac{\beta}{n}\left[b_n(\theta_{n-1})+(\psi_n(\theta_{n-1})-b_n(\theta_{n-1}))+J_n\right],
     $$
     and show that (A4.3.1), (A6.1.1)--(A6.1.7) in \cite{KushnerYin2003} are satisfied for $\tilde{\theta}$.

     From ergodicity of $Z$ we obtain that
     \begin{equation*}
       \lim_{n\to\infty}\frac{1}{n}\sum_{i=0}^{n}(b_i(\theta)-\bE_{\theta^*}[\psi_1(\theta)])=0,\quad
       \lim_{n\to\infty}\frac{1}{n}\sum_{i=0}^{n}(\psi_i(\theta)-b_i(\theta))=0,
     \end{equation*}
     which respectively imply that
     \begin{eqnarray*}
       \lim_{n\to\infty}\bP_{\theta^*}\left\{\sup_{i\geq n}\max_{0\leq t\leq \tau}\left|\sum_{j=m(i\tau)}^{m(i\tau+t)-1}\frac{\beta}{i}(b_i(\theta)-\bE_{\theta^*}[\psi_1(\theta)])\right|\geq\varepsilon\right\}=0,\\
       \lim_{n\to\infty}\bP_{\theta^*}\left\{\sup_{i\geq n}\max_{0\leq t\leq \tau}\left|\sum_{j=m(i\tau)}^{m(i\tau+t)-1}\frac{\beta}{i}(\psi_i(\theta)-b_i(\theta))\right|\geq\varepsilon\right\}=0,
     \end{eqnarray*}
     for any $\theta\in\boldsymbol\Theta$, $\varepsilon>0$ and some $\tau>0$, where $m(t)$ is the unique value of $n$ such that $\sum_{i=0}^{n-1}\frac{\beta}{i}\leq t<\sum_{i=0}^{n}\frac{\beta}{i}$.
     Therefore, (A6.1.3) and (A6.1.4) are verified. Assumption (A6.1.5) clearly holds true in our setup.
     Assumption~\eqref{eq:K2} and the fact that $b_n(\theta^*)=0$ guarantee that (A6.1.6) and (A6.1.7) are satisfied.
     Hence, according to Theorem~6.1.1 in \cite{KushnerYin2003}, the estimator $\tilde{\theta}$ converges to some limit set of the differential equation~\eqref{eq:proj-ode}.

     From R2, we see that $\bE_{\theta^*}[\psi_1(\cdot)]$ is the derivative of $\bE_{\theta^*}[\pi_1(\cdot)]$ which is a continuously differentiable real-valued function.
     Then, the limit points of \eqref{eq:proj-ode} are stationary points.
     By R3, we have that the only stationary point of \eqref{eq:proj-ode} is $\theta^*$.
     Therefore, we conclude that $\tilde{\theta}$ converges to $\theta^*$ almost surely in $\bP_{\theta^*}$.

     \hfill $\Box$

\smallskip\noindent
\textbf{Proof of Proposition~\ref{th:rootncon}.} \\
  Putting
  $V_n(\tilde{\theta}_{n-1}):=\psi_n(\tilde{\theta}_{n-1})-b_n(\tilde\theta_{n-1})$, from \eqref{eq:tre} we immediately have that
  \begin{align*}
    \Delta_n=\Delta_{n-1}+\frac{\beta}{n}b_n(\tilde\theta_{n-1})+\frac{\beta}{n}V_n(\tilde{\theta}_{n-1})+\frac{\beta}{n}J_n, \quad J_n\in-C(\tilde{\theta}_n).
  \end{align*}
  It is not hard to see that
  $$
  \|\Delta_n\|\leq\|\Delta'_n\|,
  $$
  where $\Delta'_n:=\Delta_{n-1}+\frac{\beta}{n}b_n(\tilde\theta_{n-1})+\frac{\beta}{n}V_n(\tilde{\theta}_{n-1})$.
  Hence, it is sufficient to show that
  $$
  \bE_{\theta^*}\|\Delta'_n\|=O(n^{-1}).
  $$
  The fact that $V_n(\tilde{\theta}_{n-1})$ is a martingale difference yields
  $$
  \bE_{\theta^*}\|\Delta'_n\|^2=\bE_{\theta^*}\|\Delta_{n-1}+
  \frac{\beta}{n}b_n(\tilde\theta_{n-1})\|^2
    +\frac{\beta^2}{n^2}\bE_{\theta^*}\|V_n(\tilde{\theta}_{n-1})\|^2.
  $$
  From here, applying consequently \eqref{eq:boundedE},  \eqref{eq:K2}, \eqref{eq:K1}, and noting that $b_n(\theta^*)=0$, we get
  \begin{align*}
    \bE_{\theta^*}\|\Delta'_n\|^2&=\bE_{\theta^*}\left\|\Delta_{n-1}+\frac{\beta}{n}b_n(\tilde\theta_{n-1})\right\|^2+O(n^{-2})\\
    &\leq\bE_{\theta^*}\Big[\|\Delta_{n-1}\|^2+\frac{\beta^2K_2^2}{n^2}\|\Delta_{n-1}\|^2+\frac{2\beta}{n}\Delta_{n-1}^Tb_n(\tilde\theta_{n-1})\Big]+O(n^{-2})\\
    & \leq\left(1+\frac{\beta^2K_2^2}{n^2}-\frac{2\beta K_1}{n}\right)\bE_{\theta^*}\|\Delta_{n-1}\|^2+O(n^{-2})\\
    &\leq\left(1+\frac{\beta^2K_2^2}{n^2}-\frac{2\beta K_1}{n}\right)\bE_{\theta^*}\|\Delta'_{n-1}\|^2+O(n^{-2}),
  \end{align*}
  where the last inequality holds true for large enough $n$.
Also, for any $\varepsilon>0$, and  for large enough $n$, we get
  \begin{align}\label{eq:tm2}
    \bE_{\theta^*}\|\Delta'_n\|^2\leq(1-(2K_1\beta-\varepsilon)n^{-1})\bE_{\theta^*}\|\Delta'_{n-1}\|^2+O(n^{-2}).
  \end{align}
For ease of writing, we put $p:=2K_1\beta-\varepsilon$ and $c_n:=\bE_{\theta^*}\|\Delta'_n\|^2$. Take $\varepsilon$ sufficiently small, so that $p>1$, and then chose an integer $N>p$.  Then, for $n>N$  we have by \eqref{eq:tm2} that
\begin{align*}
  c_n& \leq c_N\prod_{j=N+1}^{n} (1-\frac{p}{j})  + D_1 \sum_{j=N+1}^{n} \frac{1}{j^2}\prod_{k=j+1}^{n}(1-\frac{p}{k}) \\
  & \leq c_N\prod_{j=N+1}^{n} (1-\frac{p}{j})  + D_1 \sum_{j=N+1}^{n} \frac{1}{j^2},
\end{align*}
where $D_1$ is some strictly positive number.
Using the fact that $ \sum_{j=m}^{n}1/j^2=O(1/n) $ and $\prod_{j=m}^{n}(1-p/j)=O(1/n^p)$, for any fixed $m,p\geq1$, we immediately get that $c_n \leq O(1/n)$.
This concludes the proof.
\hfill $\Box$

\bigskip\noindent
\textbf{Proof of Theorem~\ref{th:mainth}}\\
  First, we show the quasi-asymptotic linearity of $\hat{\theta}$.
  Due to Taylor's expansion, we have that
  \begin{align}
    \frac{1}{n}\sum_{i=1}^{n}\psi_i(\theta^*)-\frac{1}{n}\sum_{i=1}^{n}\psi_i(\tilde{\theta}_{i-1})
    &=-\frac{1}{n}\sum_{i=1}^{n}\Psi_i(\tilde{\theta}_{i-1})\Delta_{i-1}
    +\frac{1}{n}\sum_{i=1}^{n}\Delta_{i-1}^T\mH\psi_i(\eta_{i-1})\Delta_{i-1} \nonumber \\
    &=:A_n+B_n, \label{eq:AnBn}
  \end{align}
  where $\eta_{i-1}, 1\leq i\leq n$, is in a neighborhood of $\theta^*$ such that $\|\eta_{i-1}-\theta^*\|\leq\|\tilde{\theta}_{i-1}-\theta^*\|$.
  Note that
  \begin{align*}
    A_n=&-\frac{1}{n}\sum_{i=1}^{n}\Psi_i(\tilde{\theta}_{i-1})\Big(\Delta_n-\sum_{j=i}^{n}\frac{\beta}{j}\psi_j(\tilde{\theta}_{j-1})\Big)\\
    =&-I_n\Delta_n+\frac{\beta}{n}\sum_{i=1}^{n}I_i\psi_i(\tilde{\theta}_{i-1})+\frac{\beta}{n}\sum_{i=1}^{n}I_iJ_i,
  \end{align*}
  and by \eqref{eq:AnBn}, we get
  $$
  I_n\Delta_n=\frac{1}{n}\sum_{i=1}^{n}\left[(\textrm{Id}  +\beta I_i)\psi_i(\tilde{\theta}_{i-1})+\beta I_iJ_i\right] -\frac{1}{n}\sum_{i=1}^{n}\psi_i(\theta^*)+B_n.
  $$
  Therefore, using the representation \eqref{eq:HthetaAlt}, we immediately have
  \begin{align}
    \hat{\theta}_n+I^{-1}(\tilde{\theta}_n)I_n\theta^*
    =&\frac{I^{-1}(\tilde{\theta}_n)}{n}\sum_{i=1}^{n}\psi_i(\theta^*)-I^{-1}(\tilde{\theta}_n)B_n.\label{eq:hatthetanstar}
  \end{align}
  Next we will show that
  \begin{align}\label{eq:limitIn-new}
 \bP_{\theta^*} \textrm{-}\!\! \lim_{n\to\infty}I_n=-I(\theta^*).
 \end{align}
 First, by \eqref{eq:K3}, we deduce that
  \begin{align*}
  \bE_{\theta^*}\Big[\frac{1}{n}\sum_{i=1}^{n}\|\Psi_i(\tilde{\theta}_{i-1})-\Psi_i(\theta^*)\|\Big]
  \leq\frac{K_3}{n} \sum_{i=1}^{n} \bE_{\theta^*}\|\Delta_{i-1}\|.
  \end{align*}
Due to Proposition~\ref{th:rootncon},
$
\frac{1}{n}\sum_{j=1}^{n}\bE_{\theta^*}\|\Delta_{i-1}\| \leq
\frac{1}{n}\sum_{j=1}^{n} j^{-1/2} = O(n^{-1/2}).
$
Hence,
\begin{equation}\label{eq:PsiMinusPsi}
\frac{1}{n}\sum_{i=1}^{n}\|\Psi_i(\tilde{\theta}_{i-1})-\Psi_i(\theta^*)\| \xrightarrow[n\to\infty]{\bP_{\theta^*}} 0.
\end{equation}
  Therefore,
  \begin{align}\label{eq:limitIn}
 \bP_{\theta^*} \textrm{-}\!\! \lim_{n\to\infty}I_n=\bP_{\theta^*}-\lim_{n\to\infty}\frac{1}{n}\sum_{i=1}^{n}\Psi_i(\tilde{\theta}_{i-1})
  =\bP_{\theta^*}\,-\,\lim_{n\to\infty}\frac{1}{n}\sum_{i=1}^{n}\Psi_i(\theta^*).
  \end{align}
 Next, observe that in view of Proposition~\ref{th:marbirk} we get

  $$
  \lim_{n\to\infty}\frac{1}{n}\sum_{i=1}^{n}\Psi_i(\theta^*)=\bE_{\theta^*}[\Psi_1(\theta^*)]=\bE_{\theta^*}[\mH \pi_1(\theta^*)]=\bE_{\theta^*}[\mH\log p_{\theta^*}(Z_0,Z_1)].
  $$
  Invoking the usual chain rule we obtain that
  $$
  \mH\log p_{\theta^*}(Z_0,Z_1)=\frac{\mH p_{\theta^*}(Z_0,Z_1)}{p_{\theta^*}(Z_0,Z_1)}-\frac{\nabla p_{\theta^*}(Z_0,Z_1)\nabla p_{\theta^*}(Z_0,Z_1)^T}{p^2_{\theta^*}(Z_0,Z_1)}=\frac{\mH p_{\theta^*}(Z_0,Z_1)}{p_{\theta^*}(Z_0,Z_1)}-\psi_1(\theta^*)\psi^T_1(\theta^*),
  $$
  so that
  \[\bE_{\theta^*}[\mH\log p_{\theta^*}(Z_0,Z_1)]= \bE_{\theta^*}[\frac{\mH p_{\theta^*}(Z_0,Z_1)}{p_{\theta^*}(Z_0,Z_1)}]-I(\theta^*).\]
  We will now show that $\bE_{\theta^*}\left[\frac{\mH p_{\theta^*}(Z_0,Z_1)}{p_{\theta^*}(Z_0,Z_1)}\right]=0.$ In fact, denote by $f_{Z_0}$ the density function of $Z_0$ under $\bP_{\theta^*}$ and in view of \eqref{eq:HP}, we have
  \begin{align*}
    \bE_{\theta^*}\left[\frac{\mH p_{\theta^*}(Z_0,Z_1)}{p_{\theta^*}(Z_0,Z_1)}\right]=&\bE_{\theta^*}\left[\bE_{\theta^*}\left[\frac{\mH p_{\theta^*}(Z_0,Z_1)}{p_{\theta^*}(Z_0,Z_1)}\middle| Z_0\right]\right]\\
    =&\int_{\bR}\bE_{\theta^*}\left[\frac{\mH p_{\theta^*}(Z_0,Z_1)}{p_{\theta^*}(Z_0,Z_1)}\middle| Z_0=z_0\right]f_{Z_0}(z_0)dz_0\\
    =&\int_{\bR}\int_{\bR}\frac{\mH p_{\theta^*}(z_0,z_1)}{p_{\theta^*}(z_0,z_1)}p_{\theta^*}(z_0,z_1)dz_1f_{Z_0}(z_0)dz_0\\
    =&\int_{\bR}\int_{\bR}\mH p_{\theta^*}(z_0,z_1)dz_1f_{Z_0}(z_0)dz_0\\
    =&\int_{\bR}\mH \int_{\bR}p_{\theta^*}(z_0,z_1)dz_1f_{Z_0}(z_0)dz_0\\
    =&\int_{\bR}(\mH 1)f_{Z_0}(z_0)dz_0=0.
  \end{align*}
  Recalling \eqref{eq:limitIn} we conclude that \eqref{eq:limitIn-new} is satisfied.

 By Assumption~R8 and strong consistency of $\tilde{\theta}$ we obtain that
  \begin{align}\label{eq:limIplus}
  \lim_{n\to\infty}I^{-1}(\tilde{\theta}_n)=I^{-1}(\theta^*) \quad \bP_{\theta^*}-a.s.,
  \end{align}
 which, combined with \eqref{eq:limitIn-new} implies that
  \begin{align}\label{eq:limthetanstar}
  -I^{-1}(\tilde{\theta}_n)I_n\theta^*\xrightarrow[n\to\infty]{\bP_{\theta^*}}\theta^*.
  \end{align}
Next, we will show that
  \begin{align}\label{eq:limbn}
    \sqrt{n}B_n\xrightarrow[n\to\infty]{\bP_{\theta^*}}0.
  \end{align}
Indeed, by \eqref{eq:K4},
$ \sqrt{n}\bE_{\theta^*}\|B_n\|
  \leq \frac{K_4}{\sqrt{n}}\sum_{i=1}^{n}\bE_{\theta^*}\|\Delta_{i-1}\|^2,
$  and consequently, in view of Proposition~\ref{th:rootncon},
  $$
  \lim_{n\to\infty}\sqrt{n}\bE_{\theta^*}\|B_n\| \leq\lim_{n\to\infty}\frac{K_4}{\sqrt{n}}\log n =0,
  $$
  which implies \eqref{eq:limbn}.

 Now, taking $\vartheta_n=-I^{-1}(\tilde{\theta}_n)I_n\theta^*$, $G_n=I^{-1}(\tilde{\theta}_n)$ and $\varepsilon_n=I^{-1}(\tilde{\theta}_n)B_n$, we deduce quasi-asymptotic linearity of $\hat{\theta}$ from \eqref{eq:hatthetanstar}, \eqref{eq:limIplus}, \eqref{eq:limthetanstar} and \eqref{eq:limbn}.

  Finally, we will show the weak consistency of $\hat{\theta}$. By ergodicity of $Z$, in view of Proposition~\ref{th:marbirk}, and using the fact that $\theta^*$ is a (unique) solution of \eqref{eq:maineq}, we have that
   $$
  \frac{1}{n}\sum_{i=1}^{n}\psi_i(\theta^*)=\bE_{\theta^*}[\psi_1(\theta^*)]=0, \quad \bP_{\theta^*}-\text{a.s.}
  $$
  Thus, $\lim_{n\to\infty}\frac{I^{-1}(\tilde{\theta}_n)}{n}\sum_{i=1}^{n}\psi_i(\theta^*)=0$   $\bP_{\theta^*}$ almost surely. This, combined with \eqref{eq:hatthetanstar}, \eqref{eq:limthetanstar} and \eqref{eq:limbn}
  implies that $    \hat{\theta}_n\xrightarrow{\bP_{\theta^*}}\theta^*$, as $n\to\infty$.
The proof is complete.

\hfill$\Box$

\bigskip\noindent
\textbf{Proof of Proposition~\ref{th:anormal}.} \\
  Let $\vartheta_n=-I^{-1}(\tilde{\theta}_n)I_n\theta^*$, $G_n=I^{-1}(\tilde{\theta}_n)$ and $I^{-1}(\tilde{\theta}_n)B_n=\varepsilon_n$. Then, property \eqref{1} follows from \eqref{eq:limthetanstar}.

  In order to prove \eqref{2}, we note that according to Theorem~\ref{th:mainth} we have
  $$
  \hat{\theta}_n-\vartheta_n=\frac{G_n}{n}\sum_{i=1}^{n}\psi_i(\theta^*)+\varepsilon_n, \quad
  \sqrt{n}\varepsilon_n\xrightarrow[n\to \infty]{\bP_{\theta^*}}0.
  $$
  Next, Proposition~\ref{prop:clt} implies that
  $$
  \frac{1}{\sqrt{n}}\sum_{i=1}^{n}\psi_i(\theta^*)\xrightarrow[n\to\infty]{d}N(0,I(\theta^*)).
  $$
 Consequently, since by \eqref{eq:limIplus} $G_n \xrightarrow{\bP_{\theta^*}}I^{-1}(\theta^*)$,  using Slutsky's theorem we get
  $$
  \frac{G_n}{\sqrt{n}}\sum_{i=1}^{n}\psi_i(\theta^*)\xrightarrow[n\to\infty]{d}N(0,I^{-1}(\theta^*)).
  $$
  The proof is complete.
\hfill $\Box$

\end{appendix}

\section*{Acknowledgments}
Part of the research was performed while Igor Cialenco was visiting the Institute for Pure and Applied Mathematics (IPAM), which is supported by the National Science Foundation.

\bibliographystyle{alpha}

{\small 
\newcommand{\etalchar}[1]{$^{#1}$}

}
\end{document}